\documentclass[12pt,reqno]{amsart}
\usepackage{geometry}                		
\usepackage{graphicx}
\usepackage{amsmath}				
\usepackage{subfigure}
\usepackage{multirow}	
\usepackage{cite}

\allowdisplaybreaks								
\usepackage{amssymb}
\usepackage{datetime}
\usepackage{color}

\newtheorem{thm}{Theorem}[section]
\newtheorem{cor}[thm]{Corollary}
\newtheorem{lem}[thm]{Lemma}

\theoremstyle{definition}

\theoremstyle{remark}
\newtheorem{rem}{Remark}[section]

\numberwithin{equation}{section}
						
\numberwithin{equation}{section}

\newcommand{\red}[1]{\textcolor{black}{#1}}
\newcommand{\blue}[1]{\textcolor{black}{#1}}

\usepackage{currfile}

\begin{document}

\title[Gaussian beams]
      {General superpositions of Gaussian beams and  propagation errors}

\author{Hailiang Liu$^\dagger$,  James Ralston$^\S$, and Peimeng Yin$^\dagger$ }
\address{$^\dagger$Iowa State University, Mathematics Department, Ames, IA 50011} \email{hliu@iastate.edu;\ pemyin@iastate.edu}
\address{$^\S$ UCLA, Mathematics Department, Los Angeles, CA 90095} \email{ralston@math.ucla.edu}



\begin{abstract}  Gaussian beams are asymptotically valid high frequency solutions
concentrated on a single curve through the physical domain, and superposition of Gaussian beams provides a powerful tool to generate more general high frequency solutions to PDEs.  We present a superposition of Gaussian beams over an arbitrary bounded set of dimension $m$ in phase space, and show that the tools recently developed in [ H. Liu, O. Runborg, and N. M. Tanushev, Math. Comp., 82: 919--952, 2013] can be applied to obtain the propagation error of order $k^{1- \frac{N}{2}- \frac{d-m}{4}}$, where $N$ is the order of beams and $d$ is the spatial dimension.  Moreover, we study the sharpness of this estimate in examples.
\end{abstract}

\maketitle

\section{Introduction}
In this paper we investigate issues related to the  accuracy of Gaussian beam approximations to high frequency wave propagation. This is related to recent results on Gaussian beam methods in 
\cite{LP15, LP17,LR09, LR10, LRT13, LRRT14,LRT16, Zh14}.  
Our model equation is the acoustic wave equation
\begin{align}\label{pp}
Pu=\partial_t^2u(x, t)  -c(x)^2\Delta u(x, t)  =0, \quad (x, t)\in \mathbb{R}_x^d \times\mathbb{R}_t
\end{align}
where $c(x)$ is a positive smooth function.  The initial data are given by
\begin{align}\label{id}
(u(x, 0), \partial_t u(x, 0))=(B_0(x), kB_1(x))e^{ikS_0(x)},
\end{align}
where $k \gg 1$ and $\nabla S_0\neq 0$, so that the data are highly oscillatory. Propagation of high frequency oscillations leads to  mathematical and numerical challenges in solving  wave propagation problems.

We study the errors which arise when one approximates solutions to the initial value problem (\ref{pp}) by superpositions of Gaussian beams. Our starting point is \cite{LRT13}, and we refer the reader to it for  more references to earlier results on superpositions of beams.   In addition, some recent effort has also been made to extend the Gaussian beam method to more complex settings such as  symmetric hyperbolic systems with polarized waves \cite{JJ15}, the Schr\"{o}dinger equation with discontinuous potentials \cite{JW14}, and wave equations in bounded convex domains \cite{BAA09, BLQ14}. 

To compare the results in \cite{LRT13} with what we do here we need to recall some conventions. For a Gaussian beam
$$v(x,t)=(a_0(x,t)+k^{-1}a_1(x,t)+\cdots+k^{-p}a_p(x,t))e^{ik\phi(x,t)}$$
we say that $v$ is an $N$th-order approximation to a solution of $Pu=0$ when the sequence of equations (from geometric optics) $L_j(x,t)=0$, $j=0,1,\dots$, holds to order $N+2-2j$ on the central ray path, where
$$[Pv](x,t) =e^{ik\phi(x,t)}\sum_{j=0}^{p+1}k^{2-j}L_j(x,t).$$
Analogously to \cite{LRT13} we use superpositions of the form
\begin{align}\label{suu}
u_{GB}(x,t)=k^{\frac{m}{2}}\int_{K_0} v(x,t;X_0)dX_0,
\end{align}
where \red{$K_0$ is a submanifold  of dimension $m$ in phase space that does not intersect $\{(x, p): \; p=0\}$}, 
 and the central ray for $v(x,t;X_0)$ has initial data $(x(0),p(0))=X_0$.  \red{In this paper we  are considering superpositions over submanifolds of $2d$-phase space of  dimension at most $d$.}  
Finally we use the unscaled energy norm
$$||u||^2_E= {1\over 2}\int_{\mathbb R^n}c^{-2}(x)|\partial_t u|^2+|\nabla_x u|^2dx$$
in place of the scaled energy norm in \cite{LRT13} which has an additional factor of $k^{-1}$. With these conventions the principal result of \cite{LRT13} becomes
\begin{thm}(\cite{LRT13})\label{1.1}
When $u(x,t)$ is the exact solution to $Pu=0$ with the initial data of the superposition $u_{GB}$ of Gaussian  beams of order $N$ over a compact subset $K_0$ of dimension $d$ in $\mathbb R^d$, the error estimate \begin{align}\label{er}
||u(\cdot,t)-u_{GB}(\cdot,t)||_E\leq C(T)k^{1-N/2}
\end{align}
holds for $t\in [0,T]$.
\end{thm} 
In this note we extend that to
\begin{thm}\label{1.2}
With the hypotheses in Theorem \ref{1.1}
\begin{align}\label{erm}
||u(\cdot,t)-u_{GB}(\cdot,t)||_E\leq C(T)k^{1-N/2-(d-m)/4}
\end{align}
when $K_0$ is a  bounded domain in phase space of dimension $m$.
\end{thm} 
\red{Comparing Theorems \ref{1.1} and  \ref{1.2} one sees that Theorem  \ref{1.1}  is the special case where $K_0$ is a domain in $\Bbb R^d$ and hence $m=d$. In Theorem  \ref{1.2} the initial data is not restricted to the \lq\lq WKB" form in (\ref{id}).  In this paper we will  always use superpositions of the form (\ref{suu}) with beams that have leading amplitudes independent of $k$. Later in this paper we sometimes fix the dependence of the error on $k$ by dividing by the energy norm of the initial data. The  decrease in the error becomes faster as $m$ decreases. This might be counter-intuitive, but it is consistent with the results in \S 5 of [9] where for a single first order beam ($N=1$ and $m=0$) in 2 dimensions $$|| u(\cdot,t)-u_{GB}(\cdot,t)||_E\leq C(T)k^{0}$$ in the unscaled energy norm above.  }

Theorem \ref{1.2} is sharp in some cases. In Section 4 we give an example with $d=3$, $m =2$ and $N=1$, where the error as a function of $k$ decays no faster than the rate in (\ref{erm}). However, the initial data in this example is not of the form (\ref{id}). A question that was left open in \cite{LRT13} is whether (\ref{er}) is sharp for data of that form. 
Numerical evidence in  \cite{LRT13} suggests that it is sharp when $N$ is even, but that when $N$ is odd the exponent on $k$ should be decreased by $1/2$, giving a faster decrease in the error as $k$ increases.
There are partial results on this conjecture. 
For superpositions of first order beams (N=1) for the semi-classical 
Schr\"odinger equation a proof of the faster decay of the error in $L^2$ is presented in \cite{Zh14}\footnote{Zheng's method can be applied to estimate errors in gaussian beam approximations for the acoustic wave equation in the $L^2$ and energy norms. This is consistent with the results in Table 2 and 4.}, based on ideas from \cite{Zh13}.

 For both the wave equation and the semi-classical Schr\"{o}dinger equation, in \cite{LRT16} the authors show that, away from caustics, the error has, uniformly, the faster decay rate in the maximum norm. However, close to caustics, their estimate degenerates. 

This paper is organized as follows: In Section 2 we derive a lower bound on the error for approximation by beam superpositions using energy conservation. In Section 3 we prove Theorem \ref{1.2}.  In Section 4 we construct the example mentioned above. In Section 5 we construct a superposition with $N=1$ for the acoustic wave equation with initial data of the form (\ref{id}) that develops a focus caustic at the origin. In a numerical study of this example we see the faster decay in the error conjectured in \cite{LRT13}. In Section 6 we construct an example in two space dimensions which develops a fold caustic on the unit circle. Here we again see numerically the faster decay conjectured in \cite{LRT13} in the energy norm, but in the maximum norm the decay is slower at some times. 
 Section 7 is concerned with initial asymptotic rates shown by the construction of various examples. \red{These examples illustrate the initial data 
that can arise from superpositions of the form (\ref{suu}) and their respective energy norms}. 
 Some final remarks are given in section 8.
\vskip.1in
 \noindent Notation: Throughout this paper,  we use the notation $A\lesssim B$ to indicate that $A$ can be  bounded by $B$ multiplied by a constant independent of the frequency parameter $k$.  $A\sim B$ stands for $A\lesssim B$ and $B\lesssim A$.

\section{Energy conservation and lower error bound}

Error estimates for Gaussian beam superpositions are based on the well-posedness of the underlying equation.  For an equation of the form
\begin{equation}\label{wave}
Pu=0,
\end{equation}
we  recall  the  well-known results here (see, e.g., \cite{LRT13})
\begin{thm} Let $u$ be an exact solution of the wave equation (\ref{wave}), and $v$ be an approximate solution of the
same problem, then we have the generic well-posedness estimate
 \begin{equation}\label{eng}
 \|(u-v)(\cdot, t_2)\|_S \leq \|(u-v) (\cdot, t_1)\|_S  +
 Ck^q \int_{t_1}^{t_2} \|Pv(\cdot, \tau)\|_{L^2}d\tau.
 \end{equation}
 These apply to both
 \begin{itemize}
 \item
 the wave equation with $P=\partial_t^2-c^2(x)\Delta$, $q=0$,  and $\|\cdot\|_S $ is the energy norm
 $$
 \|u(\cdot, t)\|_E= \left(\frac{1}{2}\int_{\mathbb{R}^d} \left( c(x)^{-2}|\partial_t u(x, t)|^2+|\nabla_x u(x, t)|^2\right) dx\right)^{1/2},
 $$
 \item and the Schr\"{o}dinger equation with $q=1$, $\epsilon=\frac{1}{k}$,
 $$
 P=-i\epsilon \partial_t + \frac{\epsilon^2}{2}\Delta
 $$
 and $\|\cdot\|_S$ is the standard $L^2$ norm.
 \end{itemize}
\end{thm}

  The lower bound on approximation errors is a consequence of the conservation law
\begin{equation}\label{en}
\|u(\cdot, t_2)\|_S=\|u(\cdot, t_1)\|_S, \forall t_1,t_2.
\end{equation}

\begin{thm}  Let $u_{GB}$ be a Gaussian beam superposition, and $u$ be an exact solution of  $Pu=0$.  
Assume that for some $\alpha>\beta >0$ there are times $t_1$ and $t_2$ and positive constants $C,\ c$ such that for $k\geq 1$
$$
Ck^{-\alpha}  \geq\|(u-u_{GB}) (\cdot, t_1)\|_S \; \text{and} \; \|(u-u_{GB}) (\cdot, t_2)\|_S \geq  c k^{-\beta},
$$
then there are  exact solutions $w_1$ and $w_2$  and a $c_0>0$ such that
$$
 \|(w_1-u_{GB}) (\cdot, t_1)\|_S=0 \hbox{ and } \|(w_1 -u_{GB}) (\cdot, t_2)\|_S\geq c_0 k^{-\beta}
$$
and
$$
 \|(w_2-u_{GB}) (\cdot, t_1)\|_S\geq c_0 k^{-\beta}  \hbox{ and } \|(w_2 -u_{GB}) (\cdot, t_2)\|_S=0
$$
for $k$ sufficiently large.
\end{thm}
\begin{proof}
Let $w_1(x, t)$ be the exact solution with data at $t=t_1$ that agree with the data of $u_{GB}(x, t)$ at $t=t_1$.
By (\ref{en}) we have
$$
\|(u-w_1)(\cdot, t_2)\|_S=\|(u-w_1)(x, t_1)\|_S= \|(u-u_{GB}) (\cdot, t_1)\|_S  \leq C k^{-\alpha}.
$$
It follows that
\begin{align*}
\|u_{GB}(\cdot, t_2) - w_1(\cdot, t_2)\|_S  & \geq  \|u_{GB}(\cdot, t_2) -u(\cdot, t_2)\|_S - \|u(\cdot, t_2)- w_1(\cdot, t_2)\|_S   \\
 & \geq  c k^{- \beta} -  C k^{- \alpha}.
\end{align*}
For the other case, we argue the following manner.  Let $w_2(x, t)$ be the exact solution with data at $t=t_2$ that agree with the data of $u_{GB}(x, t)$
at $t=t_2$.  By energy conservation we have
$$
\|(u-w_2)(\cdot, t_1)\|_S=\|(u-w_2)(x, t_2)\|_S= \|(u-u_{GB}) (\cdot, t_2)\|_S  \geq c k^{-\beta}.
$$
It follows that
\begin{align*}
\|u_{GB}(\cdot, t_1) - w_2(\cdot, t_1)\|_S  & \geq  \|u(\cdot, t_1)- w_2(\cdot, t_1)\|_S  - \|u_{GB}(\cdot, t_1) -u(\cdot, t_1)\|_S \\
 & \geq  c k^{- \beta} -  C k^{- \alpha}.
\end{align*}

\end{proof}
\begin{rem}This result may be used to identify the source of accuracy loss of the Gaussian beam superposition or other types of approximate solutions.
\end{rem}

\section{Propagation error of Gaussian beam superpositions}
Let $K_0$ be an arbitrary bounded set in phase space with dimension $m$. Given a point $X_0\in K_0$, we denote the $N$-th order Gaussian beam as $v(x, t;X_0)$, if we let $X_0$ range over $K_0$, we can form a superposition of Gaussian beams,
\begin{equation}\label{uk}
u_{GB}(x, t)=k^{m/2}\int_{K_0} v(x, t; X_0)dX_0,
\end{equation}
as an approximation to the exact solution for wave equation (\ref{wave}) with initial data $u_{GB}(x, 0)$.  

\red{
We recall that the general form of the N-th order Gaussian beam defined in \cite{LRT13} is 
$$
v(x, t; X_0)=\sum_{j=0}^{\lceil N/2\rceil -1} k^{-j}\rho_\eta(x-x(t; X_0))a_j(t, x-x(t; X_0))e^{ik \phi(t, x-x(t; X_0))},
$$
where $\rho_\eta(\cdot) \geq 0$ is a smooth cutoff function satisfying $\rho_\infty=1$ and 
$$
\rho_\eta(z)=\left\{ \begin{array}{ll}
1 &  |z| \leq \eta,\\
0 & |z| \geq 2\eta 
\end{array}
\right. \quad 0<\eta<\infty.
$$
In this construction the parameter $\eta$ is chosen as $\eta=\infty$ for the first order superposition 
 and it is taken small enough to make  $ Im(\phi(t, y)) \geq \delta |y|^2$ for $t \in [0, T]$ and 
 $|y|\leq 2\eta$  for higher order superpositions.
}
\red{For first order beams,  
$$
\phi(t, y)=S(t; X_0) +p(t; X_0)\cdot y+\frac{1}{2}y\cdot M(t; X_0)y, 
$$
associated with the first several ODEs defined by 
\begin{align*}
\dot x & =\partial_p H(x, p), \; \dot p = -\partial_x H(x, p) & (x(0), p(0))=X_0,\\
\dot S & =p\cdot \partial_p H(x, p)-H,  & S(0) =S(0; X_0),\\
\dot M & = -\partial^2_x H -M \partial^2_{xp}H -\partial_{px}^2H M-M\partial_p^2 M,\; & M(0)  =M(0; X_0).
\end{align*}
 }
\red{For equation (\ref{pp}), $H(x, p)=\pm c(x)|p|$, for which two wave modes need to be included in the superposition. We assume that 
$K_0$ does not intersect $\{(x, p)|\; p=0\}$.
No such assumption is needed for the Schr\"{o}dinger  equation with $H(x, p)=\frac{1}{2}|p|^2$.  These construction details will not be used in our error analysis, but may be helpful as a reference for reading examples constructed in sections 4-6. 
}

We now state the main result of the propagation error for superposition (\ref{uk}).

\begin{thm}\label{ee} Let $u_{GB}$ be the Gaussian beam superposition defined in  (\ref{uk}) based on $N$-th order beams \red{
emanating from a compact subset of the $m$-dimensional manifold $K_0$ in phase space}, and $u$ be the exact solution to $Pu=0$ subject to the initial data  $u_{GB}(x, 0)$,   we then have the following estimate on the propagation error,
\begin{equation}\label{result}
\|u_{GB}-u\|_S  \lesssim k^{1-N/2-(d-m)/4},
\end{equation}
where $m$ is the dimension of the domain on which initial beams are sampled, and $d$ is the spatial dimension.
\end{thm}
\begin{rem} \red{Note that operator $P$ is initially defined in (\ref{pp}), but also used for  Schr\"{o}dinger operator in Section 2. 
This theorem includes the proof of Theorem \ref{1.2},  but it is also valid for the Schr\"{o}dinger equation due to the basic estimate (\ref{eng}) and the estimate of $\|Pu_{\rm GB}\|$ to be carried out in this section.
 }
\end{rem} 
We proceed to complete the proof of this theorem by following the general steps as  in the proof of \cite[Theorem 1.1]{LRT13}.
The main difference here is that the initial set $K_0$ can be rather arbitrary in phase space.  The way that distance between
beams is measured must here be allowed to vary smoothly with the beam's initial point in phase space.

Before we outline the proof of the above result, we present a result, which shows that the accuracy of the initial approximation can be treated separately.
\begin{cor}
Let $u_{GB}$ be the Gaussian beam superposition defined in  (\ref{uk}) based on $N$-th order beams, and $u$ be the exact solution to $Pu=0$ subject to a given initial data  $u(x, 0)$,  then
\begin{equation}\label{result+}
\|u_{GB}(\cdot, t)-u(\cdot, t)\|_S  \lesssim \|u_{GB}(\cdot, 0) -u(\cdot, 0)\|_S + k^{1-N/2-(d-m)/4}.
\end{equation}
\end{cor}
\begin{proof} Let $w$ be another exact solution with initial data $u_{GB}(x, 0)$, then we have
$$
\|u_{GB}(\cdot, t)-w(\cdot, t)\|_S  \lesssim k^{1-N/2 -\frac{d-m}{4}}.
$$
The energy conservation tells that
$$
\|u(\cdot, t)-w(\cdot, t)\|_S=\|u(\cdot, 0)-w(\cdot, 0)\|_S=\|u(\cdot, 0)-u_{GB}(\cdot, 0)\|_S.
$$
These combined with the triangle inequality
$$
\|u_{GB}(\cdot, t)-u(\cdot, t)\|_S \leq \|u(\cdot, t)-w(\cdot, t)\|_S+\|u_{GB}(\cdot, t)-w(\cdot, t)\|_S
$$
lead to (\ref{result+}).
\end{proof}
 In this section, we focus only on the residual error,  where the residual can be written (following the notation of Liu, Runborg and Tenushev \cite{LRT13} and
Liu, Ralston,  Runborg and Tanushev\cite{LRRT14}) in the form
\begin{equation}\label{fres}
Pu_{GB}  = k^{m/2}\int_{K_0} [Pv(x, t; X_0)]dX_0,
\end{equation}
where $Pv(x, t;X_0)$ is a finite sum of terms of the form
\begin{align*}
f_{GB}&= k^{j}g(x, t; X_0)(x-\gamma)^{\beta}e^{ik\phi(x, t; X_0)} + O(k^{-\infty}),
\qquad
\end{align*}
with bounds
$$
{|\beta |\leq N+2,\qquad
2j \leq {2-N+|\beta|}.}
$$
Here $g$ is smooth and supported or at least bounded on
$$
\Omega(\tilde \eta, X_0):=\{x|\quad  |x- \gamma|\leq \tilde \eta\}, \quad \gamma=x(t; X_0),
$$
and $\phi$ is the $N$-th order Gaussian beam phase.   Here $\tilde \eta$ is chosen as a small number for first order beams, but can be taken as $\eta$ for higher order beams.  Moreover, $O(k^{-\infty})$ indicates terms exponentially small in $1/k$.  After neglecting these terms and
using (\ref{fres}) we can bound the $L^2$ norm of $P[u_{GB}]$ by
\begin{align*}
\|P[u_{GB}] \|^2_{L_x^2}
&\lesssim
k^m \left\|
\int_{K_0}{k^{\frac{2-N+|\beta|}2}}
e^{ik\phi}g(x-\gamma)^{\beta}
dX_0 \right\|^2_{L_x^2}\\
& \lesssim  {k^{m+ 1-N}}
\int_{\mathbb{R}^d_x} \int_{K_0}\int_{K_0} I(t, x,X_0, X_0')dX_0dX_0'dx,
\end{align*}
where the term $I$ is of the form
\begin{align*}
  I(x, t,X_0, X_0') &= k^{1+|\beta|} e^{ik \psi(x, t,X_0,X_0')}
       g(x, t; X_0')\overline{g(x, t; X_0)} \\
    &\qquad  \times  \left(x-\gamma \right)^\beta \left(x- \gamma' \right)^\beta, \quad |\beta|\leq N+2.
\end{align*}
Here
\begin{align}\label{ps}
\psi(x, t, X_0, X_0'):= \phi(x, t; X_0')-\overline{\phi(x, t; X_0)}.
\end{align}
The function $g$ and its derivatives are bounded, for $0\leq t\leq T$,
\begin{equation}\label{g}
\sup_{X_0\in K_0, x\in {\Omega(\tilde \eta;X_0)}}|\partial_x^\alpha g(x, t;X_0)| \leq  C_\alpha.
\end{equation}
The rest of this section is dedicated to establishing the
following {inequality}
\begin{equation}\label{Izz}
\left| \int_{\mathbb{R}^d_x }\int_{K_0}\int_{K_0} I(x, t,X_0, X_0') dX_0dX_0' dx \right| \lesssim k^{1-d/2-m/2}.
\end{equation}
 With this estimate  we  have
 $$
 \|P[u_{GB}]\|_{L_x^2}\lesssim k^{1-N/2 -\frac{d-m}{4}},
 $$
which together with the well-posedness estimate (\ref{eng}) leads to the desired estimate (\ref{result}).

\begin{lem}[Non-squeezing lemma]\label{nonsqueezing}
Let  $X=(x(t; X_0), p(t; X_0))$ be the Hamiltonian trajectory starting from $X_0 \in K_0$ with $K_0$ bounded.
Assume that $X(0; X_0)\in C^2(K_0)$.  Then,
\begin{align}\label{NonSqueezeIneq}
   |X_0-X_0'| \sim  |X(t, X_0) -X(t, X_0')|, \quad \forall X_0, X_0'\in K_0.
\end{align}
\end{lem}
The non-squeezing lemma \cite{LRT13} says that the distance in phase space between two smooth Hamiltonian trajectories will not shrink from its initial distance.
Here one may take any $l^p$ distance since from $X-X'=(x-x', 0)+(0, p-p')$ we have
$$
d(X, X') \leq d(x, x')+d(p, p').
$$
We recall  some main estimates from \cite{LRT13} for proving
(\ref{Izz}).

\begin{lem}[Phase estimates]\label{PhaseEst} Let $\tilde\eta$ be small and {$x\in D(\tilde\eta,X_0,X_0')$} with
$$
{D(\tilde\eta,X_0,X_0') = \Omega(\tilde\eta,X_0) \cap \Omega(\tilde\eta,X_0').}
$$
\begin{itemize}
\item For all $X_0,X_0'\in K_0$ {and sufficiently small $\tilde \eta$}, there exists a constant $\delta$ independent of $k$ such that
$$
\Im \psi \left(x, t, X_0,X_0'\right) \geq\  \delta\left[\left|x- \gamma \right|^2+\left|x-\gamma' \right|^2\right].
$$
\item For $|\gamma(x, t; X_0)-\gamma(x, t; X_0')|< \theta |X_0-X_0'|$,
\begin{align*}
 |\nabla_x\psi(x, t,X_0,X_0')| \geq C(\theta,\tilde\eta)|X_0-X_0'|,
\end{align*}
where $C(\theta,\tilde \eta)$ is independent of $x$ and positive if $\theta$ and $\tilde \eta$ are sufficiently small.
\end{itemize}
\end{lem}
Decompose $I$ as
$$
I(x, t,X_0, X_0')=I_1+I_2,
$$
with
$$
I_j=\chi_j(x, t, X_0,X_0')I(x, t,X_0, X_0'),\quad \chi_1+\chi_2 = 1,
$$
where $\chi_j(x, t, X_0,X_0')\in C^{\infty}$ is a partition
of unity such that
\begin{equation}\label{chidef}
  \chi_1(x, t, X_0,X_0') = \begin{cases} 1, & {\rm when}\ |\gamma(x, t,X_0)-\gamma(x, t,X_0')|> \theta |X_0-X_0'|,\\
 0, & {\rm when}\ |\gamma(x, t,X_0)-\gamma(x, t,X_0')|< \frac12\theta |X_0-X_0'|.
  \end{cases}
\end{equation}
We first estimate $I_1$,  which corresponds to the non-caustic region of the solution.
\begin{align*}
{\mathcal I}_1  &:=\left| \int_{\mathbb{R}^d_x}\int_{K_0}\int_{K_0} I_1(x, t,X_0, X_0') dX_0dX_0'dx \right|\\
 & \lesssim
  k^{1+|\beta|}
  {\int_{K_0}\int_{K_0}\int_{D(\eta,X_0,X_0')}}
  \chi_1|x-\gamma|^{|\beta|} |x-\gamma'|^{|\beta|}e^{-\delta k (|x-\gamma|^2+|x-\gamma'|^2)}dxdX_0dX_0'\\
     & \lesssim   k
      {\int_{K_0}\int_{K_0}\int_{D(\eta,X_0,X_0')}}
      \chi_1
   e^{-\frac{\delta k}{2}(|x-\gamma|^2+|x-\gamma'|^2)}
dxdX_0dX_0' \\
   & \lesssim   k
      {\int_{K_0}\int_{K_0}\int_{D(\eta,X_0,X_0')}}
      \chi_1
   e^{-\frac{\delta k}{4}(|x-\gamma|^2+|x-\gamma'|^2)} e^{-\frac{\delta k}{8} |\gamma-\gamma'|^2} \
dxdX_0dX_0' \\
   & \lesssim   k
      {\int_{K_0}\int_{K_0}e^{-\frac{\delta k}{8} \theta^2 |X_0-X_0'|^2} \int_{D(\tilde\eta,X_0,X_0')}}
   e^{-\frac{\delta k}{4}(|x-\gamma|^2+|x-\gamma'|^2)}  \ dxdX_0dX_0'.
\end{align*}
Here we have used the fact that
$|\gamma-\gamma'|> \theta|X_0-X_0'|$ on the support of $\chi_1$. For the inner integral over
 $D=\Omega(\tilde\eta;X_0)\cap \Omega(\tilde\eta;X_0')$, we have
\begin{align*}
{\int_{D(\tilde\eta,X_0,X_0')}}
   e^{-\frac{\delta k}{4}(|x-\gamma|^2+|x-\gamma'|^2)}dx &\leq
   \left(
   \int_{{\Omega(\tilde\eta;X_0)}}
   e^{-\frac{\delta k}{2}(|x-\gamma|^2)}dx
   \int_{{\Omega(\tilde\eta;X_0')}}
   e^{-\frac{\delta k}{2}(|x-\gamma'|^2)}dx\right)^{1/2} \\
   & \lesssim k^{-d/2}.
\end{align*}
From this it follows that
\begin{equation}\label{I1x}
|{\mathcal I}_1|\lesssim k^{(2-d)/2}
      \int_{K_0}\int_{K_0}
   e^{-\frac{\delta k}{8} \theta^2 |X_0-X_0'|^2} \ dX_0dX_0'.
\end{equation}
Letting $\Lambda =\sup_{X_0,X_0'\in K_0}|X_0-X_0'|<
\infty$ be the diameter of $X_0$, we have
\begin{align*}
|{\mathcal I}_1| & \lesssim  k^{(2-d)/2}
      \int_{K_0}\int_{K_0}
   e^{-\frac{\delta k}{8} \theta^2 |X_0-X_0'|^2} \ dX_0dX_0'\\
 & \lesssim k^{(2-d)/2}\int_0^\Lambda \tau^{m-1}e^{-\frac{k \delta\theta^2}{8}\tau^2} d\tau\\
&\lesssim  k^{1-d/2-m/2},
\end{align*}
which concludes the estimate of ${\mathcal I}_1$.

In order to estimate ${\mathcal I}_2$ we use a version of
the non-stationary phase lemma.
\begin{lem}[Non-stationary phase lemma]\label{statphase}
Suppose that $u(x;\zeta)\in C_0^\infty(\Omega
\times Z)$, where $\Omega$ and $Z$ are compact
sets and ${\psi(x; \zeta)\in C^\infty(O)}$ for
some open neighborhood $O$ of $\Omega \times Z$.
If $\nabla_x \psi$ never vanishes in $O$, then
for any $K=0,1,\ldots$,
\begin{align*}
   \left|\int_\Omega u(x; \zeta)e^{i k\psi(x; \zeta)}dx \right|
   \leq C_K k^{-K}  \sum_{|\alpha|\leq K}\int_\Omega \frac{|\partial_x^{\alpha}u(x; \zeta)|}{|\nabla_x\psi(x;\zeta)|^{2K-|\alpha|}}
   e^{- k \Im \psi(x; \zeta)}dx\ ,
\end{align*}
where $C_K$ is a constant independent of $\zeta$.
\end{lem}
We now define
\begin{align*}
{\mathcal I}_2 &:= \int_{\mathbb{R}^d_x} I_2(x, t,X_0,X_0')dx\\
&= k^{1+|\beta|}
  \int_{{D(\tilde \eta,X_0,X_0')}}\chi_2
  e^{ik \psi(x, t,X_0,X_0')} g(x, t; X_0') \overline{g(x, t;X_0)}  (x-\gamma)^\beta (x-\gamma')^\beta dx.
\end{align*}
Non-stationary phase Lemma \ref{statphase} can be applied to ${\mathcal I}_2$
with $\zeta=(X_0,X_0')\in K_0 \times K_0 $ to give,
\begin{align*}
 \left|{\mathcal I}_2 \right|
 &\lesssim k^{1+|\beta|-K}  \sum_{|\alpha| \leq K} \int_{{D(\tilde \eta,X_0,X_0')}}\frac{\left|\partial^\alpha_x \left[(x-\gamma)^\beta(x-\gamma')^\beta \chi_2 g' \overline{g}\right]\right|}{|\nabla_x\psi(t, x,X_0,X_0')|^{2K-|\alpha|}}e^{-\Im k \psi(t, x,X_0,X_0')} dx \\
&  \lesssim  k^{1-d/2}\sum_{|\alpha|\leq K}\frac{1}{(|X_0-X_0'|\sqrt{k})^{2K-|\alpha|}}.
\end{align*}
On the support of $\chi_2$ the difference $|X_0-X_0'|$ can be arbitrary small, in which case
this estimate is not useful.  Following \cite{LRT13}, we use the fact that  the estimate is true also for
$K=0$ so that ${\mathcal I}_2 $ can be bounded by the minimum of the $K=0$ and $K>0$ estimates.
Therefore,
\begin{align*}
\left|{\mathcal I}_2  \right|
&\lesssim  k^{1-d/2}\min\left[1,\sum_{|\alpha|\leq K}\frac{1}{\left(|X_0-X_0'|\sqrt{k}\right)^{2K-|\alpha|}}\right] \\
& \lesssim \frac{k^{1-d/2}}{1+\left(|X_0-X_0'|\sqrt{k}\right)^{K}} \ .
\end{align*}
Finally, letting $\Lambda =\sup_{X_0,X_0'\in
K_0}|X_0-X_0'|< \infty$ be the diameter of
$K_0$, we compute
\begin{align*}
\int_{K_0} \int_{K_0} \left| {\mathcal I}_2 \right| dX_0dX_0'
&\lesssim k^{\frac{2-d}{2}}\int_{K_0 \times K_0} \frac{1}{1 + \left(|X_0-X_0'|\sqrt{k}\right)^K} dX_0dX_0'\\
&\lesssim k^{\frac{2-d}{2}}\int_0^\Lambda \frac{1}{1 + (\tau \sqrt{k})^K}\tau^{m-1} d\tau \\
& \lesssim k^{\frac{2-d-m}{2}} \int_0^\infty \frac{\xi^{m-1}}{1+\xi^K}d\xi \\
& \lesssim  k^{\frac{2-d-m}{2}} \ ,
\end{align*}
if we take $K>m$. This shows the ${\mathcal I}_2$
estimate,  which proves claim (\ref{Izz}).

\section{Example of a Gaussian beam superposition}
Let $r=|x|$, $x\in \mathbb{R}^3$. Then for any smooth function $f$,
$$
u(x,t)=(f(t-r)-f(t+r))/r
$$
satisfies $\partial_t^2 u=\Delta u$. Take
$f(r)=\exp(-ikr-kr^2/2)/k$. Then
$$
u(x,0)=2i{\sin(kr)\over kr} e^{-kr^2/2}\hbox{ and } \partial_t u(x,0)=2\left({\sin(kr)\over r}+\cos(kr)\right)e^{-kr^2/2}.
$$
The exact solution here is a highly oscillatory spherical wave which concentrates on $r=|t|$ as $k\to\infty$. The Cauchy data of this solution at $t = 0$ can be approximated very well by a superposition of Gaussian beams.

Note that $$\int_{S^2}e^{ikx\cdot\omega}d\omega = 4\pi {\sin(kr)\over kr}, $$  since the integral is a radial solution of $\Delta w+k^2w=0$,  which equals $4\pi$ at $x=0$, then we have
$$
u(x,0)=\int_{\mathbb{S}^2}v(x,0;\omega)d\omega,
$$
where
\begin{equation}
v(x,0;\omega)={i\over 2\pi}\exp(ikx\cdot\omega-k|x|^2/2).
\end{equation}
Let us approximate $u(x,t)$ by a superposition of beams
$$
u_{GB}(x,t)=\int_{\mathbb{S}^2}v(x,t;\omega)d\omega.
$$
 Hence $u_{GB}(x,0)=u(x,0)$. It will turn out, somewhat surprisingly, that $\partial_t u_{GB}(x,0)$ is very close to $\partial_t u(x,0)$. In fact, the first order Gaussian beam can be explicitly given as
 $$
 v(x,t;\omega)=a(t)e^{ik\phi(x,t;\omega)},
 $$
 where
$$
\phi(x,t;\omega)=x\cdot\omega-t+{i\over 2} \left( (x\cdot \omega-t)^2+{1\over 1+it}(|x|^2-(x\cdot\omega)^2)\right),
$$
and $2\pi a(t)=i (1+it)^{-1}$.  Note that $ \partial_t v =(ik \partial_t \phi a+\partial_t a)e^{ik\phi}, \hbox{ so }$
$$
\partial_ t v(0,x; \omega)={k\over 2\pi} \left( 1+ix\cdot\omega+{1\over 2}(|x|^2-(x\cdot \omega)^2)+\frac{1}{k} \right)e^{ikx\cdot\omega-k|x|^2/2}.
$$
Now we can compute
\begin{align*}
\partial_t u_{GB}(x,0)& =\int_{\mathbb{S}^2} \partial_t v(x,0;k,\omega)d\omega \\
& ={k\over 2\pi} e^{-kr^2/2} \left(1 + \frac{d}{dk} + \frac{r^2}{2} +\frac{1}{2} \frac{d^2}{dk^2} + \frac{1}{k} \right) \int_{\mathbb{S}^2}e^{ikx\cdot\omega}d\omega.
\end{align*}
Using
$$
{d\over dk}\int_{\mathbb{S}^2}e^{ikx\cdot\omega}d\omega = 4\pi \left({\cos(kr)\over k}-{\sin(kr)\over k^2r}\right),
$$
we have
\begin{align*}
 \partial_t u_{GB}(x,0)& =2k \bigg( {\sin(kr)\over kr}+ ({\cos(kr)\over k}-{\sin(kr)\over k^2r}) +\frac{r^2}{2} {\sin(kr)\over kr} \\
& \qquad  +\frac{1}{2} {d\over dk}\left({\cos(kr)\over k}-{\sin(kr)\over k^2r}\right) + {\sin(kr)\over k^2r}\bigg)e^{-kr^2/2} \\
& =2 \bigg({\sin(kr)\over r}+\cos(kr) - \frac{\cos(kr)}{k}+{\sin(kr)\over k^2r})\bigg)e^{-kr^2/2}.
\end{align*}
Note that the first two terms in that expression equal $\partial_t u(x,0)$. 
To estimate the data we use the standard energy norm $||(u, \partial_t u)||_E^2=\int_{\Bbb R^3} |\partial_t u|^2+|\nabla_xu|^2dx$.
In that norm the difference of the initial data satisfies
$$||(u(0),\partial_t u(0))-(u_{GB}(0), \partial_t u_{GB}(0)||_E\sim k^{-7/4}\hbox{, but }||(u(0), \partial_t u(0))||_E\sim k^{-1/4}.$$
So the relative error in the initial data is $O(k^{-3/2})$.

Now we get to the main point: How large is $u(x,t)-u_{GB}(x,t)$? We need to compute
$$
u_{GB}(x,t)={i\over2\pi( 1+it)}\int_{\mathbb{S}^2}e^{ik\phi(x,t;\omega)}d\omega.
$$
Introducing spherical coordinates so that $x\cdot\omega=|x|\cos \rho$ and $d\omega =\sin \rho d\rho d\phi$ with the domain of integration $0\leq \rho\leq\pi$ and $0\leq \phi\leq 2\pi$ and setting $|x|=r$,  this becomes - after substituting $s=\cos \rho$
$$
u_{GB}(x,t)={i\over 1+it}\int_{-1}^1 e^{ik\phi(s)}ds,
$$
where
$$
\phi(s)=[rs-t+tr^2(2+2t^2)^{-1}(1-s^2)] +{i\over 2}[(rs-t)^2+r^2(1+t^2)^{-1}(1-s^2)].
$$
Note that, for $t>0$,  the real part of the exponent in the integrand is strictly negative unless $s=1$ and $r=t$. Moreover, for $t>0$ and $r$ in a sufficiently small neighborhood of $t$ the maximum of the real part of exponent for $-1\leq s\leq 1$ is assumed at $s=1$. So we can find $u_{\rm GB}(x,t)$, up to terms of order $k^{-1}e^{-k(r-t)^2/2}$, by using the leading term in the integration by parts expansion: Choosing $\rho$ with support near $s=1$ and $\rho(1)=1$,
$$\int_{-1}^1e^{ik\phi(s)}\rho(s)ds=\int_{-1}^1{d\over ds}(e^{ik\phi(s)}){\rho(s)\over ik\phi^\prime(s)}ds ={e^{ik\phi(1)}\over ik\phi^\prime(1)} -\int_{-1}^1e^{ik\phi}{d\over ds}\left({\rho(s)\over ik\phi^\prime(s)}\right)ds.$$
One continues this expansion by repeated integration by parts. In particular, the integral term on the right is $O(k^{-2}e^{-k(r-t)^2/2})$. Since
 $\phi(\pm 1)=\pm r-t +i(r \mp t)^2/2$ and
$$
{1\over \phi^\prime(1)}={1\over r}\left({1+t^2\over 1-it+ (r-t)(-t+it^2)}\right),
$$
hence for $t>\delta>0$ and $r$ close to $t$,
\begin{align*}
u_{GB}(x,t)-u(x,t) & =  {e^{ik\phi(1)}\over k (1+it) \phi^\prime(1)} - \frac{1}{kr} (e^{ik\phi(1)} -e^{ik\phi(-1)}) +O\left({1\over k^2}e^{-k(r-t)^2/2}\right) \\
  & ={1\over kr}\bigg({t(r-t)\over 1+t(t-r)}\bigg)e^{ik(r-t)-k(r-t)^2/2}+O\left({1\over k^2}e^{-k(r-t)^2/2}\right).
\end{align*}
At this point we want to obtain a lower bound on 
$||u_{GB}(\cdot,t)-u(\cdot, t)||_E$. The dominant terms in the first 
derivatives of $u_{GB}(\cdot,t)-u(\cdot, t)$ come from the factor 
$\exp(ik(r-t))$ and bring down a factor of $k$. So, letting $s=r-t$, 
this leaves a dominant term which is a nonvanishing multiple of 
$se^{-ks^2/2}$, and hence has $L^2$ norm bounded below by a multiple of 
$k^{-3/4}$. That implies $||u_{GB}(\cdot, t)-u(\cdot, t)||_{E} \sim k^{-3/4}$.
However, here the Gaussian beam superposition is missing a factor of $k$ compared to Theorem \ref{1.2}.  Hence this example shows that Theorem \ref{1.2}  is sharp when $d=3$, $m=2$ and $N=1$.

\section{An example for the 3D acoustic wave equation}
This will be the construction of a Gaussian beam superposition for the initial value problem
\begin{align}\label{6.0}
\partial_t^2  u-\Delta u=0,\ u(x, 0)=a(|x|)e^{ik|x|},\          \partial_t u(x,0)=0,\ (x,t)\in \Bbb R^3_x\times\Bbb R_t,
\end{align}
where $a(r)=0$ in a neighborhood of $r=0$. From here on $|x|=r$ will be used.

The exact solution to this initial value problem is
$$
u(r,t)={1\over r}(f(t+r)-f(t-r))\hbox{ where }f(s)={sa(s)\over 2}e^{iks}
$$
extended to $\Bbb R$ by $f(s)=-f(-s)$.
Note that for $t>0$
$$
u(0,t)=\hbox{lim}_{r\to 0}{f(t+r)-f(t-r)\over r}=(ikta(t)+a(t)+ta'(t))e^{ikt},
$$
and the solution has a strong peak at $r=0$, when $t$ is in the support of $a$. We want to see the effect of this caustic.

Following the \lq\lq standard procedure" for first order beams, the Gaussian beam superposition will be
\begin{equation}\label{(1)}
u_{GB}(x,t)=\left({k\over 2\pi}\right)^{3/2}\int_{\Bbb R^3}A^+(t; y)e^{ik\Phi^+(x,t; y)}+ A^-(t; y)e^{ik\Phi^-(x,t; y)}dy,
\end{equation}
where $\Phi^+(x,0; y)=\Phi^-(x,0; y)$, $\partial_t \Phi^+(x,0; y)=- \partial_t \Phi^-(x,0; y)$ and $A^+(0; y)=A^-(0; y)$.
So we have two families of Gaussian beams
$$v^\pm(x,t; y)=A^\pm(t; y)e^{ik\Phi^\pm(x,t; y)},$$
where both phases $\Phi^\pm$ are based on the initial phase $S(x)=|x|$, but the $v^\pm$ are concentrated on the rays $(x(t),t)=(y\pm ty/|y|,t)$; \red{see e.g., \cite[superposition (3.1)]{LR09}}.

From here on we will often use $y=s\omega$, $|\omega|=1$. Again the standard construction gives
$$\Phi^\pm(x,t; y)=x\cdot \omega\mp t+{1\over 2}(x-(s\pm t)\omega)\cdot M(\pm t; y)(x-(s\pm t)\omega ),$$
where $M(0; y)=(1/|y|)P_{\omega^\perp}+iI$
 and $\partial_t M+MP_{\omega^\perp}M=0$.
Here $I-P_{\omega^\perp}$ is the orthogonal projection on the span of $\omega$. A modest amount of computation shows
$$
M(t; y)=b(t; s)P_{\omega^\perp} +i I,$$
where
$$
b(t; s)={1-it (1+is)\over (s+ist+t)}.
$$
So
$$
\Phi^\pm(x,t; y)=x\cdot \omega\mp t+{b(\pm t,s)\over 2}(|x|^2-(x\cdot \omega)^2)+{i\over 2}(|x|^2-2(s\pm t)x\cdot \omega +(s\pm t)^2).
$$
The amplitudes $A^\pm$ are given by
$$A^\pm (t; s)={a(s)\over 2}(1\pm t(s^{-1}+i))^{-1}.
$$
Since $x$ appears in $v^\pm$ only as $|x|$ and $x\cdot\omega$, we have $ u_{GB}(x,t)=w(r,t)$. This can be seen by integrating 
in spherical coordinates.   Also $\partial_tu_{GB}(x,0)=0$.

Now we need to determine the order of $||u(\cdot,t)-u_{GB}(\cdot,t)||_E$. 
The contributions to $u_{GB}$ from $\int A^+(t; y)\exp(ik\Phi^+(x,t; y))dy$ will be concentrated at $x=(t+s)\omega$, and, since $s\geq 0$ and we consider $t>0$, they will be negligible near $x=0$. Hence we  will omit that term from all formulas from here on.  Let $v={x\cdot\omega\over |x|}=\cos(\theta)$.  While this is undefined at $x=0$, substitution of $v$ for $\theta$ in (\ref{(1)}) leads to an integral in spherical coordinates that is well-behaved as $x\to 0$. Namely
\begin{align}\label{ugbn} \notag 
u_{GB}(x,t)= & 2\pi\left({k\over 2\pi}\right)^{3/2}\int_0^\infty \! A^-(t,s) s^2ds\! \int_{-1}^1dv \exp(ikCrv-ikDr^2v^2) \\
& \times \exp(ik(t+Dr^2)-k(r^2+(s-t)^2)/2),
\end{align}
where $C=1-i(s-t)$ and $D=b(-t,s)/2$.  
Presumably one could evaluate this formula further, but that is a daunting calculation. Instead we offer the numerical results in the next section.

\subsection{Numerical results} 
In this section and in the numerical results in Section 6.1 we will use relative norms to estimate errors, i.e. norms scaled by the 
corresponding norm of the beam superposition. \red{In these examples that has the effect of decreasing the power of $k$ in the energy norm by one, and leaving the power unchanged in the $L^2$-norm, but Example 5 in section 7 shows that this does not always happen.}
Since the energy norm of the initial data is of order $k$ in both cases and $m=d$, Theorem  \ref{1.2} predicts a relative error of order $k^{-1/2}$ for first order beams. We will see that the actual error is numerically of order $k^{-1}$ as conjectured in \cite{LRT13}.  

We take $a(s)=4(s-r_0)^4(s-r_1)^4$ for $r_0 \leq s \leq r_1$; $a(s)=0$ otherwise, here $r_0=0.1, \ r_1=1.0$. The evaluation of (\ref{(1)})
is done using $80\times 80$ meshes of  $[r_0, r_1]\times [-1,1]$ and $5^2=25$ quadrature points in each element using the reduced integral 
(\ref{ugbn}) and its counter part with $t$ replaced by $-t$.  At the focus $x=0$, the results are reported at Table \ref{tab3d005}, in which the errors are calculated by $e_k=|u-u_{GB}|/|u_{GB}|$, and the orders of convergence are obtained by
\begin{equation}\label{eoc}
\text{EOC}=\log_2 \left( \frac{e_k}{e_{2k}}\right).
\end{equation}
\begin{table}[!htbp]\tabcolsep0.03in
\caption{3D Gaussian beam single point errors and orders of convergence.}
\begin{tabular}[c]{||c|c|c|c|c|c|c|c||}
\hline
  \multirow{2}{*}{$t$} & k=320 & \multicolumn{2}{|c|}{k=640} & \multicolumn{2}{|c|}{k=1280} & \multicolumn{2}{|c||}{k=2560}  \\
\cline{2-8}
 & error & error & order & error & order & error & order\\
\hline
 0.4 & 0.109724 & 0.064004 & 0.78 & 0.0347248 & 0.88 & 0.0177462 & 0.97  \\
 0.55 & 0.0820207 & 0.0420894 & 0.96 & 0.0213659 & 0.98 & 0.0118253 & 0.85  \\
 0.7 & 0.0822853 & 0.0418797 & 0.97 & 0.0211195 & 0.99 & 0.0102407 & 1.04  \\
\hline
\end{tabular}\label{tab3d005}
\end{table}
We also test the energy errors and orders of convergence at some $t$ in $(0, 1)$. The error in energy norm is calculated by
$
e_k = {\|u-u_k\|_E}/{\|u_k\|_E  },
$
with $ \|v\|_E^2 = \frac{1}{2}\int_{\mathbb{R}^3} |u_t|^2 + |\nabla_x u|^2 dx$, evaluated over the ball of radius $r_1+t$.
The errors and orders of convergence using  (\ref{eoc}) are reported in Table \ref{tab3deng}.
\begin{table}[!htbp]\tabcolsep0.03in
\caption{3D Gaussian beam energy errors and orders of convergence.}
\begin{tabular}[c]{||c|c|c|c|c|c|c|c||}
\hline
\multirow{2}{*}{$t$} & k=320 & \multicolumn{2}{|c|}{k=640} & \multicolumn{2}{|c|}{k=1280} & \multicolumn{2}{|c||}{k=2560}  \\
\cline{2-8}
 & error & error & order & error & order & error & order\\
\hline
 0.4 & 0.111507 & 0.0671302 & 0.73 & 0.0354353 & 0.92 & 0.0185048 & 0.94  \\
 0.5 & 0.0716308 & 0.0388652 & 0.88 & 0.0193636 & 1.01 & 0.00994688 & 0.96  \\
 0.55 & 0.0825064 & 0.0429692 & 0.94 & 0.0213441 & 1.01 & 0.0108103 & 0.98  \\
 0.7 & 0.0834459 & 0.0427814 & 0.96 & 0.0211234 & 1.02 & 0.0106206 & 0.99  \\
 0.8 & 0.0458945 & 0.0242241 & 0.92 & 0.0100285 & 1.27 & 0.0053213 & 0.91  \\
\hline
\end{tabular}\label{tab3deng}
\end{table}
The numerical results with the gain in the order of convergence agree with the above asymptotic estimate.  


The mechanism that leads to a relative error of order $k^{-1}$ in this example is probably the cancelation of terms of order $k^{-1/2}$ in the Gaussian beam superposition in (\ref{ugbn}). This is the result of the spherical symmetry in this superposition.

\section{A 2D example with fold caustics}
This section is devoted to the construction of a Gaussian beam superposition with fold caustics for the 2D acoustic wave equation 
\begin{align}\label{2de}
\square_{x,t} u:=\partial_t^2  u-\Delta u=0.
\end{align}
Let us  consider a gaussian beam superposition with ray paths given by 
 $$
 (x_1(t;\theta,s),x_2(t;\theta,s)), 
$$
 where 
\begin{align*}
& x_1(t;\theta,s)=\sqrt 2 \cos(\theta +\pi/4)+(t+s)\sin(\theta), \\ 
& x_2(t;\theta,s)=\sqrt 2\sin(\theta+\pi/4)-(t+s)\cos(\theta). 
\end{align*}
These rays are tangent to the unit circle at $(x_1,x_2)=(\cos\theta, \sin\theta)$ 
and propagating in the direction of the tangent $(\sin\theta,-\cos\theta)$. That defines the parameter $\theta$. The parameter $s$ is distance along the ray path, chosen so that $s=0$ on $x_1^2+x_2^2=2$ and $s=1$ on $x_1^2+x_2^2=1$.  More precisely the relation between $r$ and $s$ is (by the Pythagorean Theorem)
$$1+(s-1)^2=r^2$$
or  $s=1- \sqrt{r^2-1}$ for  $s< 1$ and $s=1+ \sqrt{r^2-1}$  for $s>1$. 
The phase function associated with these ray paths, which was complicated in euclidian coordinates, is quite simple in $(\theta,s)$. 
It can be chosen as 
$$
S(x_1(0;\theta,s),x_2(0;\theta,s))=-\theta+s,
$$  
defined for $0\leq s<1$, and $-\pi<\theta<\pi$. The function $\exp(ikS(x_1,x_2))$ will be single-valued on the annulus bounded by the circles of radius $1$ and $\sqrt 2$ only when $k$ is an integer. In the numerical examples we will take $k$ to be an integer.

The Hessian of $S(x_1,x_2)$ has to be a multiple of the orthogonal projection $P^\perp$ onto $(\cos\theta,\sin\theta)$, the vector perpendicular to the ray path. 
We find that at $(x_1(0;\theta,s),x_2(0;\theta,s))$, 
\begin{align*}
\left(\begin{array}{cc}
\partial^2_{x_1x_1} S & \partial^2_{x_1x_2}S\\ 
\partial^2_{x_1x_2}S&\partial^2_{x_2x_2}S 
\end{array}\right)
=\frac{1}{s-1}\left(
\begin{array}{cc}
\cos^2\theta  & \cos\theta\sin\theta  \\
\sin\theta\cos\theta  &      \sin^2\theta
\end{array}
\right)
=\frac{1}{s-1}P^\perp(\theta).
\end{align*}
The Hessian of the phase in the Gaussian beam, $M(t;\theta,s)$, has to be given by 
$$
M=a(t,s)P(\theta)+b(t,s)P^\perp(\theta),
$$ 
where $P=I-P^\perp$, $\partial_t a=0$, $\partial_t b+b^2=0$, and $(a(0,s),b(0,s))=(i,i+1/(s-1))$. 
So the Hessian of the phase is a lot like the Hessian in the 3D example. 
In fact, we have
$$
b(t,s)={b(0,s)\over 1+tb(0,s)}={i+1/(s-1)\over 1+it+t/(s-1)}={1+i(s-1)\over t+s-1+it(s-1)}.
$$
The next step in the construction would be to find the amplitude, but for that one needs the phase. That is
\begin{align*}
\phi(x,t;\theta,s)& =-\theta +s+(x-x(t;\theta,s))\cdot(\sin\theta,-\cos\theta)\\
& \qquad +{1\over2}(x-x(t;\theta,s))\cdot M(t;\theta,s)(x-x(t;\theta,s))\\
& =-\theta -t +1 +x\cdot (\sin\theta,-\cos\theta)+{1\over2}(x-x(t;\theta,s))\cdot M(t;\theta,s)(x-x(t;\theta,s)).
\end{align*}
That comes from formulas (1.5) and (1.6) in \cite{LRT13} 
with one small observation: the function $\phi_0(t;z)$ with $z=(\theta, s)$ does not depend on $t$. You can see that from the fact that since  the Hamitonaian $H$ is homogeneous of degree one in $p$, $\dot x(t;z)\cdot p(t;z)=H(t,x(t;z),p(t;z))$, which forces $\dot \phi_0(t;  z)=0$ (see also equation (3.10c) in \cite{LRT16}).

Continuing, we have $\partial_t\phi =-1$ and $\square_{x,t} \phi=-b(t,s)$ when $x=x(t;\theta,s)$. So the solution of the transport equation, $2A_t\phi_t+(\square_{x,t} \phi)A=0$ is just 
$$
A(t;\theta,s)=A(0;\theta,s)(1+tb(0;s))^{-1/2}.
$$ 
So the complete Gaussian beam superposition will be
\begin{align}\label{ugb}
u_{GB}(x,t)={k\over 2\pi}\int_0^1ds\int_0^{2\pi}A(t;\theta,s)e^{ik\phi(x,t;\theta,s)}(1-s)d\theta,
\end{align}
where $1-s$ is the absolute value of the Jacobian of $(x_1(0;\theta,s),x_2(0,\theta,s))$ with respect to $(\theta,s)$.
\red{The  contributions from beams built with  the other choice,  $\partial_t\phi=1$,  propagate away from the disk $\{|x|\leq 1\}$ as $t$ 
increases, and are negligible near the caustics on the circle. Hence we  have omitted  those contributions from all formulas and numerical 
results below.}
In the next  section we will examine  the accuracy of the method numerically.

\subsection{Numerical results} 

In addition to estimates of accuracy in the energy norm,  we will also give  numerical estimates in the maximum norm.  \red{ The results in \cite{LRT16}  restricted to first order beams with $O(1)$ initial data show  that  $||u_{GB}(t)-u(t)||_{L^\infty}\leq Ck^{-1}$ away from caustics, see \cite[estimate (6.1)]{LRT16}. For domains including caustics \cite{LRT16} gives the  weaker estimate  $||u_{GB}(t)-u(t)||_{L^\infty}\leq k^{1/2}$, and these estimates 
hold in relative norms well.}
\red{Our numerical results in Table 3 below show that at caustics the order of error is 
$$\|u-u_k\|_{L^\infty}\leq Ck^{-\alpha(t) }\|u(\cdot, 0)\|_{L^\infty}
$$ 
with $\alpha(t)$ varying in $(0.5, 1)$. We see that the numerical order of error near caustics is greater than the error away from caustics but much smaller than the bound in \cite{LRT16}}. 

We consider the  2D acoustic wave equation (\ref{2de}) on $[0, T] \times \Omega$, where $\Omega=[-L/2,L/2]^2$ with $L=4$, subject 
to initial data $(u, \partial_t u)(x, 0)=(u_{GB}, \partial_t u_{GB})(x, 0)$ and  periodic boundary conditions. For the Gaussian beam superposition (\ref{ugb}) we take initial amplitude
\begin{equation*}
A(0; \theta, s)=
\left \{
\begin{array}{rl}
    (s-s_0)^2(s-s_1)^2, & \quad s_0 \leq s \leq s_1,\\
    0, & \quad \text{otherwise},
\end{array}
\right.
\end{equation*}
which is supported on $1+(1-s_1)^2 \leq x_1^2 +x_2^2 \leq 1+(1-s_0)^2$ for  $s_0, s_1 \in (0, 1)$, since  $1+(s-1)^2=r^2$.

We use the fast Fourier transform to approximate the ``exact solution", and use it to determine the errors in the Gaussian beam superposition. 
For $K$ large enough, say $K=1024$, we partition $\Omega$ by a uniform rectangular mesh $\Omega=\left[-L/2:h:L/2-h\right]^2$, with $h=L/K$.
 We obtain the ``exact solution" and its derivatives numerically using Matlab 2018a in the following steps.
\begin{itemize}
  \item Step 1 (Initial preparation) We calculate the integrals in the Gaussian beam superposition $u_{GB}(x, t)$; more specifically $u_{GB}(x,0), (u_{GB})_t(x, 0)$ using \texttt{integral2} with absolute tolerance $10^{-8}$. 
  \item Step 2 (fast Fourier transform) The Fourier transform of (\ref{2de}) gives 
  \begin{equation}\label{wavefft}
\begin{aligned}
\partial_t^2 \widehat{u} = & i^2 (\kappa_1^2+\kappa_2^2)\widehat{u},\\
\widehat{u}(0) = & \widehat{u}_{GB}(\kappa_1,\kappa_2, 0), \\
\partial_t \widehat{u}(0) = & (\partial_t\widehat{u_{GB}})(\kappa_1,\kappa_2, 0), 
\end{aligned}
\end{equation}
where $(\kappa_1, \kappa_2) \in \left[\frac{2\pi}{L}(0,1, \cdots, K/2-1, -K/2, -K/2+1, \cdots, -1)\right]^2$ is adopted in the fast Fourier transform 
(\texttt{fft2}) in Matlab. 
  \item Step 3 (Solving ODE) The exact solution of (\ref{wavefft}) is determined  by
\begin{equation*}
\widehat{u} =
\left \{
\begin{array}{rl}
    \partial_t \widehat{u}(0) t + \widehat{u}(0), & \quad \text{if } \kappa_1=\kappa_2 =0, \\
    \widehat{u}(0)\cos( \sqrt{\kappa_1^2+\kappa_2^2}t ) +    \frac{\partial_t \widehat{u}(0)}{\sqrt{\kappa_1^2+\kappa_2^2} } \sin( \sqrt{\kappa_1^2+\kappa_2^2}t ), & \quad \text{otherwise}.
\end{array}
\right.
\end{equation*}
  \item Step 4 (inverse fast Fourier transform) We obtain the ``exact solution" $u$ and its derivatives $\partial_t u, \partial_{x_1}u, \partial_{x_2}u$ through the inverse fast Fourier transform (\texttt{ifft2} in Matlab) applied to $\widehat{u}, \partial_t \widehat{u}, i\kappa_1\widehat{u}, i\kappa_2\widehat{u}$, respectively.
\end{itemize}

We test the case $s_0=0.25, \ s_1=0.75$.  With this choice, the wave propagates within the entire computational domain $\Omega$ for $t\leq T=0.8$, and caustics appear only for $0.25=1-s_1<t<1-s_0=0.75$.  

\red{In this example, $u(x, 0)=u_k(x, 0)$.
A refined numerical test indicates that 
$$
 \|u_k(\cdot, t)\|_{L^\infty} \sim k^{\beta(t)} \|u(\cdot, 0)\|_{L^\infty},
$$ 
where  the rate $\beta(t)$, shown in Figure \ref{GBRate}, is calculated over $N\times N$ meshes with $N=2.5 \times 10^5$, 
and of frequencies $k=40960$ and $k=81920$.  
From this figure, we see that $\beta(t) \sim 0$ when away from caustics, but $\beta(t)$ can go up to about $1/6$ in the presence of caustics.}
\begin{figure}
	\centering
	\subfigure{\includegraphics[width=0.49\textwidth]{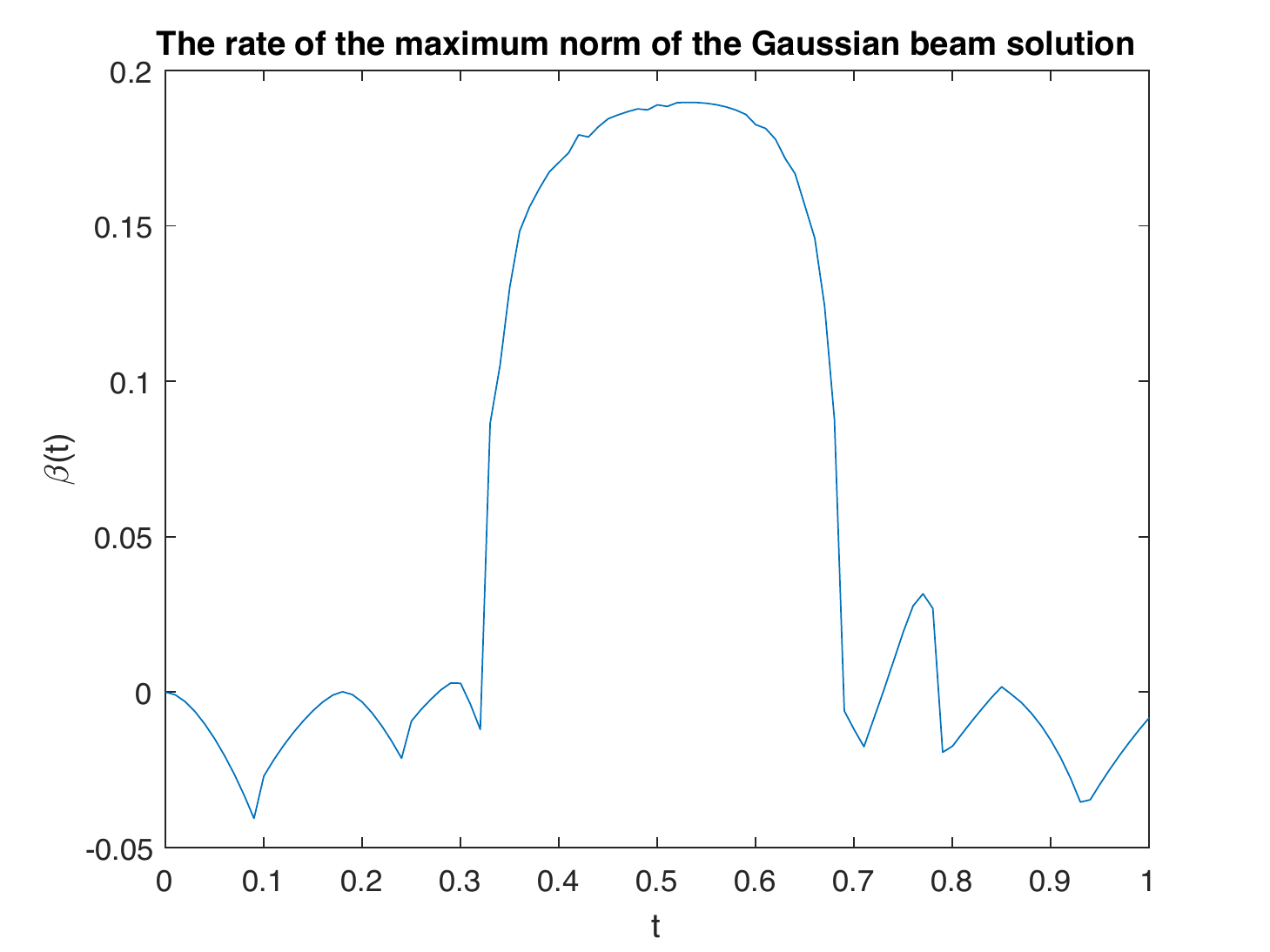}}
	\caption{The rate of $\|u_k(\cdot,t)\|_{L^\infty} / \|u(\cdot,0)\|_{L^\infty}$ of the Gaussian beam solution.}
    \label{GBRate}
\end{figure}
The experimental orders of convergence (EOC) are obtained by
\begin{equation}\label{eoc+}
\text{EOC}=\log_2 \left( \frac{e_k}{e_{2k}}\right),
\end{equation}
where $e_k$ is the relative error between the ``exact solution" $u(x, t)$ and $u_k: =u_{GB}$.  \\

\noindent\textbf{Test case 1. Convergence in $L^\infty$ norm.} \\
We first check the $L^\infty$ errors and orders of convergence from $t=0.15$ to  $t=0.8$.  The error in $L^\infty$ norm is approximated by
$$
e_k = \frac{\|u(\cdot, t)-u_k(\cdot, t)\|_{L^\infty}}{\|u(\cdot, 0)\|_{L^\infty}  }, \quad \|v\|_{L^\infty} = \max_{(x,y)\in \Omega}(|v|).
$$
From the errors and orders of convergence reported in Table \ref{tab2dlinf} obtained using $1024 \times 1024$ meshes, \red{we find that the orders of accuracy are decreased in the present of fold caustics.}  \red{At $t=0.15$ and $0.80$, the orders of accuracy are increased when $k$ increases to $320$, much closer to the desired first order since caustics are not present.
} \\

\begin{table}[!htbp]\tabcolsep0.03in
\caption{$L^\infty$ errors and orders of convergence of 2D Gaussian beam superposition.}
\begin{tabular}[c]{||c|c|c|c|c|c||} 
\hline
\multirow{2}{*}{$t$} & k=80 & \multicolumn{2}{|c|}{k=160} & \multicolumn{2}{|c|}{k=320} \\
\cline{2-6}
 & error & error & order & error & order \\ 
\hline
0.15 & 0.0134258 & 0.00775526 & 0.79 & 0.00390146 & 0.99 \\
\hline
0.30 & 0.0349802 & 0.0205824 & 0.77 & 0.0119712 & 0.78 \\
\hline
0.40 & 0.0439910 & 0.0253134 & 0.80 & 0.0155627 & 0.70 \\ 
\hline
0.42 & 0.0469892 & 0.0241403 & 0.96 & 0.0155593 & 0.63 \\
\hline
0.45 & 0.0514442 & 0.0258481 & 0.99 & 0.0147372 & 0.81 \\
\hline
0.50 & 0.0578684 & 0.0296832 & 0.96 & 0.0161166 & 0.88 \\
\hline
0.60 & 0.0694715 & 0.0371298 & 0.90 & 0.0184097 & 1.01 \\
\hline
0.70 & 0.0795142 & 0.0437080 & 0.86 & 0.0229725 & 0.93 \\ 
\hline
0.80 & 0.0878869 & 0.0485395 & 0.86 & 0.0254761 & 0.93 \\
\hline
\end{tabular}\label{tab2dlinf}
\end{table}

\noindent\textbf{Test case 2. Convergence in energy norm.} We next check the energy errors and orders of convergence from $t=0.15$ to $0.8$. 
The error in energy norm is approximated by
$$
e_k = \frac{\|u-u_k\|_E}{\|u_k\|_E  }, \quad \|v\|_E^2 := \frac{1}{2}\int_{\Omega} |u_t|^2 + |\nabla_x u|^2 dx.
$$
The results in Table \ref{tab2deng} show that $1$st order of accuracy in energy norm is obtained regardless of the appearance of caustics. \red{Note that in contrast to the maximum norm, for energy norm $\|u_k(\cdot, t)\|_E=\|u(\cdot, 0)\|_E$. } 
\begin{table}[!htbp]\tabcolsep0.03in
\caption{Energy errors and orders of convergence of 2D Gaussian beam superposition.}
\begin{tabular}[c]{||c|c|c|c|c|c||}
\hline
\multirow{2}{*}{$t$} & k=80 & \multicolumn{2}{|c|}{k=160} & \multicolumn{2}{|c||}{k=320}  \\
\cline{2-6}
 & error & error & order & error & order \\
\hline
 0.15 & 0.0138 & 0.0071 & 0.96 & 0.0034 & 1.06  \\
\hline
0.30 & 0.0274 & 0.0141 & 0.96 & 0.0068 & 1.05  \\
\hline
0.40 & 0.0364 & 0.0187 & 0.96 & 0.009 & 1.06  \\
\hline
0.42 & 0.0382 & 0.0196 & 0.96 & 0.0094 & 1.06  \\
\hline
0.45 & 0.0408 & 0.021 & 0.96 & 0.0101 & 1.06  \\
\hline
0.50 & 0.0452 & 0.0233 & 0.96 & 0.0112 & 1.06  \\
\hline
0.60 & 0.054  & 0.0278 & 0.96 & 0.0134 & 1.05  \\
\hline
0.70 & 0.0626 & 0.0323 & 0.95 & 0.0156 & 1.05  \\
\hline
0.80 & 0.071  & 0.0367 & 0.95 & 0.0177 & 1.05  \\
\hline
\end{tabular}\label{tab2deng}
\end{table}

The gain in order of accuracy in the energy norm indicates the contribution from cancellations of first order beams, this is consistent with the numerical evidence  in \cite{LRT13}.  However, the order of accuracy in $L^\infty$ norm can vary in time due to the presence of caustics; 
while when away from caustics the uniform first order of accuracy in  $L^\infty$ norm has been proven in  \cite{LRT16}.

\section{Examples of general superpositions}
In this section we discuss the growth rate in $k$ of general superpositions,
$$
u_{GB}(x)=k^{m/2} \int_{K_0} v(x; X_0)dX_0,
$$
when measured in the energy norm.  This will depend on the detailed description of $K_0$, and we discuss by examples. \red{To simplify presentation, we only estimate the $L^2$-norm of $\nabla _x u$ in all examples, instead of computing the whole energy norm. For beams the $L^2$-norm of the spatial gradient is always comparable to the $L^2$ norm of the initial time derivative.
}

Let $K_0$ be parameterized  by $z \in \Sigma$ so that
$$
K_0=\{(x, p)|\quad x=x(z), p=p(z), \quad  z\in \Sigma \subset  \mathbb{R}^m\}.
$$
Here listed are some typical  examples. \\

\noindent{\bf Example 1. } If the data is concentrated at one point (say, in the case of a point source for stationary problems), one may
consider
$$
K_0=\{X=(x, p)| \quad x(z)=0, \quad p(z)=z \in \mathbb{S}^{d-1}\}, \quad m=d-1.
$$
In the example presented in \S 4, we have
$$
u_{GB}(x,0)={ ki\over 2\pi} \int_{\mathbb{S}^2}\exp(ikx\cdot\omega-k|x|^2/2) d\omega.
$$
This corresponds to $d=3$ and $m=2$ with
$$
K_0=\{(0, \omega), \quad \omega \in \mathbb{S}^2\}.
$$
The asymptotic rate of its energy norm  is
$$
||u_{GB}(\cdot, 0)||_E \sim k^{3/4}=k^{1-\frac{d-m}{4}}.
$$
We may also consider the case  $m=d$ with
$$
K_0=\{(x, p), \quad x=0, \quad  p=z \in \mathbb{R}^d\}.
$$
Let $a(p)$ be a smooth function compactly supported in $p$, and
$$
u_{GB}(x,0)=  \frac{k^{m/2}}{(2\pi)^{d/2}}\int_{\mathbb{R}^d}a(p) \exp(ikx\cdot p - k|x|^2/2) dp.
$$
Hence $u_{GB}(x, 0)=k^{m/2} \hat a(kx)e^{-k|x|^2/2}$, and
\begin{align*}
\|\partial_x u_{GB}(\cdot, 0)\|^2_{L^2} & \sim k^m \int_{\mathbb{R}^d_x} k^2 \sum_{j=1}^d |\partial_{y_j} \hat a(kx)-x_j \hat a(kx)|^2e^{-k|x|^2}dx \\
& \sim k^{2+m}\int_{\mathbb{R}^d_y} \sum_{j=1}^d |\partial_{y_j} \hat a(y)-y_j \hat a(y)/k|^2e^{-|y|^2/k}dy\sim k^{2 +m-d}.
\end{align*}
This implies  $\|u_{GB}(\cdot, 0)\|_E\sim k^{1-\frac{d-m}{4}}$.  This together with the result in Theorem \ref{ee} says that the relative error is no greater than $k^{-N/2}$, as we expected.
\\

\noindent{\bf Example 2.  }
A more general example of a superposition.  \\
Let $z=(z^{(1)},z^{(2)})$ where $z^{(1)}=(z_1,...,z_r)\in \Bbb R^r$ and $z^{(2)}=(z_{r+1},..,z_m)\in \Bbb R^{m-r}$. Consider the superposition of Gaussian beams in $\Bbb R^d$
$$u_{GB}(x)=k^{m/2}\int_{\Bbb R^m}a(z)e^{ikx^{(1)}\cdot z^{(1)}-(k/2)(|x^{(1)}|^2+|x^{(2)}-z^{(2)}|^2+|x^{(3)}|^2)}dz.$$
Here $x^{(3)}=(x_{m+1},...,x_d)\in \Bbb R^{d-m}$.
We will take $a(z)=e^{-|z|^2/2}$ to make some computations explicit. So $a(z)$ {\it nearly} has compact support.  We have
$$u_{GB}(x,0)=k^{m/2}(2\pi)^{r/2} e^{-k(|x^{(1)}|^2 +|x^{(3)}|^2)  /2-k^2|x^{(1)}|^2/2}\int_{\Bbb R^{m-r}}e^{-|z^{(2)}|^2/2-k|x^{(2)}-z^{(2)}|^2/2}dz^{(2)}.$$
Since
\begin{align*}
|z^{(2)}|^2+k|x^{(2)}-z^{(2)}|^2 & = (1+k)|z^{(2)}|^2-2kz^{(2)}\cdot x^{(2)}+ k|x^{(2)}|^2 \\
& = |(1+k)^{1/2}z^{(2)}- k(1+k)^{-1/2}x^{(2)}|^2+k(1+k)^{-1}|x^{(2)}|^2,
\end{align*}
then
$$
u_{GB}(x,0)=k^{m/2}(2\pi)^{m/2}(1+k)^{(r-m)/2} e^{-k(|x^{(1)}|^2+  |x^{(3)}|^2) /2-k^2|x^{(1)}|^2/2-k(1+k)^{-1}|x^{(2)}|^2/2}.
$$
We have
$$
\nabla u_{GB}(x,0)=-(k(1+k)x^{(1)},k(1+k)^{-1}x^{(2)}, kx^{(3)})u_{GB}(x,0).
$$
This  gives
\begin{align*}
k^{-m} ||\nabla u_{GB}(\cdot,0)||^2_{L^2}
& = c_1k^2 (1+k)^{2+r-m} (k(1+k))^{-1-r/2}(k/(1+k))^{-(m-r)/2}k^{-(d-m)/2} \\
& \qquad +c_2k^2 (1+k)^{-2+r-m}(k(1+k))^{-r/2}(k/(1+k))^{-1-(m-r)/2}k^{-(d-m)/2} \\
& \qquad +c_3k^2 (1+k)^{r-m}(k(1+k))^{-r/2}(k/(1+k))^{-(m-r)/2}k^{-1-(d-m)/2}\\
&= c_1k^{1-d/2}(1+k)^{1-m/2} +c_2k^{1-d/2}(1+k)^{-1-m/2} +c_3k^{1-d/2}(1+k)^{-m/2} .
\end{align*}
where $c_1$, $c_2$ and $c_3$ are powers of $2\pi$. The first term in that expression dominates, and we have
$$||\nabla u_{GB}(\cdot,0)||^2_{L^2}\sim k^{2-(d-m)/2} \quad \hbox{ or }\quad ||\nabla u_{GB}(\cdot,0)||_{L^2}\sim k^{1-(d-m)/4}.
$$
Assuming that the $L^2$-norm of $\partial_t u_{GB}(x,0)$ is of the same order, we can compare that with
$||u(\cdot,t)-u_{GB}(\cdot,t)||_E$ for which we have the estimate (for $|t|<T$)
$$||u(\cdot,t)-u_{GB}(\cdot,t)||_E\leq Ck^{1/2-(d-m)/4}$$
for first order beams, and get the relative error estimate
$$
||u(\cdot,t)-u_{GB}(\cdot,t)||_E/||u_{GB}(\cdot,0;k)||_E\leq k^{-1/2}.
$$
This shows what can happen when initial data is not of form (\ref{id}).


\noindent{\bf Example 3. }
For wave equation (\ref{pp}) subject to the WKB initial data,
$$
(u(x, 0), \partial_t u(x, 0))=(A_0(x, k), B_0(x, k))e^{ikS_0(x)},
$$
compactly supported in $\Omega \subset \mathbb{R}^d$, one may consider $m=d$ with
$$
K_0=\{(x, p), \quad  x \in \Omega:= {\rm supp}(A_0)\cup {\rm supp}(B_0),  \;p=\nabla_x S_0(x)\}.
$$
The superposition of the first order Gaussian beam is given by
$$
u_{GB}(x, 0)= \frac{k^{m/2}}{2}\int_{\Omega} A_0(x_0)e^{ik\phi(x, 0;x_0)}dx_0,
$$
where
$$
\phi(x, 0; x_0)=S_0(x_0)+ p_0\cdot (x-x_0) + \frac{1}{2} (x-x_0)\cdot M_0 (x-x_0),
$$
with $p_0=\nabla_x S_0(x_0)$ and $M_0=\partial_x^2 S_0 (x_0) +iI$.
Note that
$$
\partial_x u_{GB}(x, 0) \sim \frac{ik^{1+m/2}}{2}\int_{\Omega} A_0(x_0)(p_0 + M_0(x-x_0))e^{ik\phi(x, 0;x_0)}dx_0.
$$
Hence the energy norm can be estimated as
$$
\|\partial_x u_{GB}(\cdot, 0)\| \lesssim k^{1+m/2} \left\| \int_{\Omega} (1+|x-x_0|)e^{-k|x-x_0|^2/2}dx_0 \right\| \lesssim k^{1-\frac{d-m}{4}}=k.
$$
This upper bound is as expected. \\

\noindent{\bf Example 4. }   For the WKB data $e^{(ik-1)|x|^2/2}$, we consider
$$
u_{GB}(x,0)=k^{d/2} \int_{\Bbb R^d}e^{ik|x|^2/2-(k/2)|x-z|^2}e^{-|z|^2/2}dz.
$$
Note that $|x|^2/2= |z|^2/2 +z\cdot (x-z) +|x-z|^2/2$, this superposition corresponds to the case with $p(z)=z$,  $x(z)=z$,  initial phase $S_0(x)=|x|^2/2$, and initial amplitude $e^{-|x|^2/2}$.   Since
\begin{align*}
|z|^2+k|x-z|^2 & = (1+k)|z|^2- 2kz\cdot x+ k|x|^2 \\
& =(1+k)|z- k(1+k)^{-1}x|^2+k(1+k)^{-1}|x|^2,
\end{align*}
we have
\begin{align*}
u_{GB}(x, 0) & =k^{d/2} e^{-(k/(2k+2))|x|^2 +ik|x|^2/2 }\int_{\Bbb R^d}e^{-(1/2)(1+k)|z- k(1+k)^{-1}x|^2}dz \\
& =k^{d/2} \left({2\pi\over k+1}\right)^{d/2}\exp(-(k/(2k+2))|x|^2 +ik|x|^2/2 ).
\end{align*}
This implies
$$||u_{GB}(\cdot,0)||_{L^2}\sim k^{0},$$
and passing to $\nabla u_{GB}(x,0)$ brings down a factor  of order $k$. Hence,
$$
 ||\nabla u_{GB}(\cdot,0)||_{L^2}\sim k^{1}.
$$
We may also consider
$$
u_{GB}(x,0)= k^{d/2} \int_{\Bbb R^d}e^{ik|x|^2/2 - (k/2)|x-z|^2}a(z)dz,
$$
where $a$ is assumed to be smooth with compact support.
This corresponds to the case with $p(z)=z$,  $x(z)=z$,  initial phase $S_0(x)=|x|^2/2$, and initial amplitude $a(x)$, for
$$
|x|^2/2 =|z|^2/2 + z\cdot (x- z) +|x-z|^2/2.
$$
We have
\begin{align*}
u_{GB}(x, 0) & =k^{d/2} e^{ik|x|^2/2 }\int_{\Bbb R^d}e^{-(k/2)|z- x|^2}a(z)dz, \\
\partial_x u_{GB}(x, 0) & = k^{1+d/2} e^{ik|x|^2/2 }\int_{\Bbb R^d} (ix -(x-z))e^{-(k/2)|z- x|^2}a(z)dz.
\end{align*}
This  implies
\begin{align*}
  ||u_{GB}(\cdot,0)||_{L^2} & \lesssim \|a\|_{L^2},  \\
    ||\partial_x u_{GB}(\cdot,0)||_{L^2} & \lesssim k^{1} \|x a\|_{L^2}+ k^{1/2-d/2}\|a\|_{L^2}  \lesssim k^{1} =k^{1-\frac{d-m}{4}}.
\end{align*}

\noindent{\bf Example 5.  } This example is a bit surprising.  \\
Let
$$
u_{GB}(x,0)=k^{d/2} \int_{\Bbb R^d}e^{ikx\cdot z-(k/2)|x-z|^2}e^{-|z|^2/2}dz.
$$
In other words $p(z)=z$ and $x(z)=z$.
In this case it is easy to compute $u_{GB}(x,0)$.
Since
\begin{align*}
|z|^2+k|x-z|^2 & = (1+k)|z|^2- 2kz\cdot x+ k|x|^2 \\
& =(1+k)|z- k(1+k)^{-1}x|^2+k(1+k)^{-1}|x|^2,
\end{align*}
we have
\begin{align*}
u_{GB}(x, 0) & =k^{d/2} e^{-(k/(2k+2))|x|^2}\int_{\Bbb R^d}e^{ikx\cdot z-(1/2)(1+k)|z- k(1+k)^{-1}x|^2}dz \\
&=k^{d/2}  (1+k)^{-d/2} \exp(-k/(2k+2))|x|^2+ik^2(1+k)^{-1}|x|^2) \int_{\Bbb R^d}e^{i\frac{k}{\sqrt{k+1}}x \cdot \xi}e^{-|\xi|^2/2}d\xi\\
& =k^{d/2} \left({2\pi\over k+1}\right)^{d/2}\exp(ik^2(1+k)^{-1}|x|^2-k/2|x|^2).
\end{align*}
We can see this implies
$$||u_{GB}(\cdot,0)||_{L^2}\sim k^{-d/4},$$
and passing to $\nabla u_{GB}(x,0)$ brings down factors of $x_j$ multiplied by factors of order $k$. Hence,
$$
||\nabla u_{GB}(\cdot,0)||_{L^2}\sim k^{1/2-d/4}.
$$
Note that here $|| u_{GB}(\cdot,0)||_{E}$ is not of order $k$.  However, like Example 2, the initial data here is not of form (\ref{id}).

We may consider a more general case in the form
$$
u_{GB}(x,0)= k^{d/2} \int_{\Bbb R^d}e^{ikx\cdot z-(k/2)|x-z|^2}a(z)dz,
$$
where $a$ is assumed to be smooth with compact support.  Let $\eta =z - x$ and the integral becomes
$$
u_{GB}(x, 0)=k^{d/2} e^{ik|x|^2}\int_{\Bbb R^d}e^{ikx\cdot \eta -(k/2)|\eta|^2}a(\eta + x)d\eta.
$$
Using the Plancherel Theorem one can write
\begin{align*}
\int_{\Bbb R^d}e^{ikx\cdot \eta-(k/2)|\eta|^2}a(\eta+x)d\eta
& =C\int_{\Bbb R^d} k^{-d/2} e^{-|kx-\xi|^2(2k)^{-1}}\hat a(\xi)e^{ix\cdot \xi}d\xi.
\end{align*}
Now, assuming that $a(z)$ is smooth with compact support, $|\hat a(\xi)|\leq C_N(1+|\xi|^2)^{-N}$ for all $N$. So
$$|u_{GB}(x,0)|\leq A_N \int_{\Bbb R^d}e^{-|kx-\xi|^2(2k)^{-1}} (1+|\xi|^2)^{-N}d\xi.$$
Now divide that integral into $I_1=\int_{\{|\xi|<k|x|/2\}}$ and $I_2=\int_{\{|\xi|>k|x|/2\}}$. Then, taking $N$ large enough that $\int_{\Bbb R^d}(1+|\xi|^2)^{-N}d\xi<\infty$,   the contribution to $|u_{GB}(x,0)|$ from $I_1$  is bounded by
$$B_N  e^{-k|x|^2/8},$$
and the contribution from $I_2$ is bounded by
$$I_0(x)=B_N  \int_{\{|\xi|>k|x|/2\}}(1+|\xi|^2)^{-N}d\xi.$$
Finally we split $I_0$ into $\chi_{\{|x|>k^{-1/2}\}}(x)I_0(x)+\chi_{\{|x|<k^{-1/2}\}}(x)I_0(x)$ ($\chi_E$ is the characteristic function of $E$) .
Using that splitting and taking $N$ sufficiently large ($N=N(M)$), one ends up with for any $M>0$ and $\alpha>1$,
$$
|u_{GB}(x,0)|\leq [C_Mk^{-M}(1+|x|^\alpha)^{-d} + B_N\chi_{\{|x|<k^{-1/2}\}}(x) ]+B_Ne^{-k|x|^2/8}.
$$
 That leads once more to
$$||u_{GB}(\cdot,0)||_{L^2}\sim k^{-d/4}.$$

\section{Final remarks}
We have presented results on superpositions of Gaussian beams of order $N$ in dimension $d$ over arbitrary bounded sets of dimension $m$ in phase space, and shown that the error in the approximation of the exact solution with the same initial data is $O(k^{1-N/2 - (d-m)/4})$ in energy norm. This result is sharp for general super-positions. For exact solutions with WKB initial data, i.e. initial data of the form (\ref{id})  our numerical evidence  in the case $N = 1$ and $d = m$  indicates the stronger estimate $O(1)$, or $O(k^{-1})$ in the relative energy norm as conjectured in \cite{LRT13}. However, the numerical estimates in maximum norm are not uniform in time due to the presence of caustics; while away from caustics we know the relative propagation error in maximum norm  is $O(k^{-1})$ as has been proven in \cite{LRT16}.

\section*{Acknowledgments} 
This work was supported by the National Science Foundation under Grant RNMS (Ki-Net) 1107291 and by NSF Grant DMS1812666.


\begin{thebibliography}{10}

\bibitem{BAA09}
S. Bougacha, J.L. Akian, and R. Alexandre.
\newblock   Gaussian beams summation for the wave equation in a convex domain.
\newblock {\em Commun. Math. Sci.}, 7(4):973--1008, 2009.

\bibitem{BLQ14}
G. Bao, J.  Lai, and J.-L. Qian. 
\newblock Fast multiscale Gaussian beam methods for wave equations in bounded convex domains. 
\newblock  {\em J. Comput. Phys.}  261:36--64, 2014. 

\bibitem{JJ15}
L. Jefferis and S.  Jin. 
\newblock A Gaussian beam method for high frequency solution of symmetric hyperbolic systems with polarized waves. 
\newblock  {\em Multiscale Model. Simul.} 13(3):733--765, 2015.

\bibitem{JW14}
S. Jin, D.-M.  Wei, and  D.-S. Yin. 
\newblock Gaussian beam methods for the Schr\"{o}dinger equation with discontinuous potentials. 
\newblock  {\em J. Comput. Appl. Math.}  265:199--219, 2014.

  \bibitem{LP15}
 H. Liu and M. Pryporov.
 \newblock  Error estimates of the Bloch band-based Gaussian beam superposition
for the Schr\"{o}dinger equation.
  \newblock {\em Contemp. Math.}, 640, 87--114, 2015.

  \bibitem{LP17}
 H. Liu and M. Pryporov.
 \newblock Error estimates for Gaussian beam methods applied to symmetric strictly hyperbolic systems.
\newblock {\em Wave Motion}, 73:57--75, 2017.

\bibitem{LR09}
 H. Liu and J. Ralston.
 \newblock Recovery of high frequency wave fields for the acoustic wave equation.
\newblock {\em Multiscale Model. Simul.}, 8: 428--444, 2009.

\bibitem{LR10}
H. Liu and J. Ralston.
\newblock Recovery of high frequency wave fields from phase spaceÐbased measurements
\newblock  {\em Multiscale Model. Simul.}, 8(2):622--644, 2010.


\bibitem{LRT13}
H. Liu, O. Runborg, and N. M. Tanushev.
\newblock Error estimates for Gaussian beam superpositions.
\newblock  Math. Comp., 82: 919--952, 2013.

\bibitem{LRRT14}
H. Liu, O. Runborg, J. Ralston and N. M. Tanushev.
\newblock Gaussian beam methods for the Helmholtz equation
\newblock  {\em SIAM J. Appl. Math.},  74(3): 771--793, 2014.

\bibitem{LRT16}
H. Liu, O. Runborg and N. Tanushev.
\newblock Sobolev and max norm error estimates for Gaussian beam superpositions.
 \newblock  {\em  Comm. Math. Sci.},  14(7): 2041--2076, 2016.

\bibitem{Zh13}
C.-X. Zheng.
\newblock Global geometrical optics method.
\newblock  {\em Commun. Math. Sci.},  11(1): 105--140, 2013.

\bibitem{Zh14}
C.-X. Zheng.
\newblock Optimal error estimates for first--order Gaussian beam approximations to the Schr\"{o}dinger equation.
\newblock  {\em SIAM J. Numer. Anal.},  52(6): 2905--2930, 2014.

\end{thebibliography}
\end{document}